\documentclass[10pt,twocolumn,twoside]{IEEEtran}
\ifCLASSINFOpdf
\else
\fi
\usepackage{amsmath,amssymb,amsthm}
\usepackage{xcolor}
\usepackage{graphicx}
\usepackage{booktabs}
\usepackage[hidelinks]{hyperref}
\usepackage{algorithm}
\usepackage{mathabx}
\usepackage{multirow}

\usepackage{footnote}
\makesavenoteenv{tabular}
\makesavenoteenv{table}
\usepackage[font=small]{caption}
\usepackage{subcaption}

\usepackage[style=ieee,backend=bibtex,url=false,doi=false,isbn=false,date=year,citestyle=numeric-comp,maxbibnames=10,minbibnames=10,maxcitenames=10,mincitenames=10]{biblatex}
\bibliography{./paper}

\newtheorem{definition}{Definition}

\newtheorem{theorem}{Theorem}

\newtheorem{lemma}[theorem]{Lemma}

\newtheorem{assumption}{Assumption}

\newtheorem{remark}{Remark}

\begin{document}
%
\title{Decomposition of non-convex optimization via bi-level distributed ALADIN}
%
%
%

\author{Alexander~Engelmann,~\IEEEmembership{Student Member,~IEEE,}
        Yuning~Jiang,~\IEEEmembership{Student Member,~IEEE,}
        Boris~Houska,~\IEEEmembership{Member,~IEEE,}
        and Timm~Faulwasser,~\IEEEmembership{Member,~IEEE}
	\thanks{This work received funding from the European Union’s Horizon 2020
	research and innovation program under grant agreement No. 730936. TF acknowledges further
	support from the Baden-W\"u{}rttemberg Stiftung under the Elite Programme for Postdocs. 
	 YJ and BH are supported by ShanghaiTech University, Grant-Nr. F-0203-14-012.}
 \thanks{AE and TF are with the Institute for Automation and Applied Informatics, Karlsruhe Institute of Technology, Eggenstein-Leopoldshafen, Germany
 	{\tt alexander.engelmann@kit.edu timm.faulwasser@ieee.org}}%
\thanks{YJ and BH are with the School of Information Science and Technology, ShanghaiTech University, Shanghai, China
	{\tt $\{$jiangyn, borish$\}$@shanghaitech.edu.cn }}%

}

\maketitle

\begin{abstract}
Decentralized optimization algorithms are important in different contexts, such as distributed optimal power flow or distributed model predictive control, as they avoid central coordination and enable decomposition of large-scale problems.
In case of constrained non-convex optimization only a few algorithms are currently are available; often their performance is limited,  or they lack convergence guarantees.
This paper proposes a framework for decentralized non-convex optimization via bi-level distribution of the Augmented Lagrangian Alternating Direction Inexact Newton (ALADIN) algorithm. 
Bi-level distribution means that the outer ALADIN structure is combined with an inner distribution/decentralization level solving a condensed variant of ALADIN's convex coordination QP by  decentralized algorithms. 
We prove sufficient conditions ensuring local convergence while allowing for inexact decentralized/distributed solutions of the coordination QP. 
Moreover, we show how a decentralized variant of conjugate gradient or decentralized ADMM schemes can be employed at the inner level. 
We draw upon case studies from power systems and robotics to illustrate the performance of the proposed framework. 
\end{abstract}


%
\IEEEpeerreviewmaketitle

\section{Introduction}

\emph{Distributed} optimization algorithms are of interest in many engineering applications due to their ability to solve large-scale problems efficiently and enable solution to optimization problems with limited information exchange \cite{Boyd2011}.\footnote{Note that there is no unified notion of \emph{distributed optimization}---while the classical optimization literature allows for (preferably small) central coordination \cite{Bertsekas1989}, in the power systems community optimization with any kind of centralized coordination is called \emph{hierarchical} or hierarchically distributed \cite{Molzahn2017}.} 
%
These algorithms often employ a (usually simple) global coordination step while computationally expensive operations are executed in parallel or decentralized by local agents.
Some algorithms avoid any kind of central coordination and communicate on a neighborhood basis only; they are commonly denoted as \emph{decentralized} \cite{Nedic2018a}.
Decentralized algorithms are of significant application interest; yet they are difficult to design and to analyze. 

The majority of results on distributed optimization investigates \emph{convex} problems \cite{Boyd2011,Bertsekas1989,Gabay1976}. 
Many practically relevant problems, however,  are inherently \emph{non-convex}; examples range from non-linear model predictive control \cite{Stewart2011,Mehrez2017} to power systems \cite{Erseghe2014a,Engelmann18a,Kekatos2013} and wireless sensor networks \cite{Lee2006}.

An approach to unconstrained non-convex problems via a push-sum algorithm can be found in \cite{Tatarenko2017}; \cite{Hours2017}  employs an alternating trust-region method with convergence guarantees for general non-convex problems.
Algorithms based on distributing steps of centralized algorithms like Sequential Quadratic Programming (SQP) can be found in \cite{Tran-Dinh2013,Necoara2009a}.
A decomposition method of the linear algebra subproblems  of an interior point method using the Schur complement is presented in \cite{Kang2014}.
Moreover, for special classes of non-convex problems, the Augmented Direction of Multipliers Method (ADMM) has convergence guarantees \cite{Hong2016,Wang2019}.
Note, however, that only a few algorithms for \emph{decentralized} non-convex optimization exist; to the best of our knowledge the only currently available algorithms are decentralized variants of the before mentioned algorithms \cite{Hong2016} and \cite{Tatarenko2017}.

The present paper proposes a design framework for general purpose decentralized algorithms applicable to constrained non-convex optimization defined over networks with generic topology. 
To this end, we build upon the Augmented Lagrangian Alternating Direction Inexact Newton (ALADIN) method \cite{Houska2016} which  solves general non-convex problems to local optimality with guarantees.
ALADIN exhibits advantageous local quadratic convergence under mild technical assumptions; however it requires solving a centralized Quadratic Program (QP) as coordination step. 

Specifically, we propose to \emph{decentralize} ALADIN by solving the coordination QP---which is the only centralized step---in a decentralized fashion.\footnote{Note that the globalization routine in ALADIN also requires central coordination. However, as our goal for the present paper is  developing a local algorithm, the globalization routine is not considered. Hence decentralizing the solution of the coordination QP provides an avenue towards a fully decentralized algorithm.}
To this end, we apply condensing techniques  similar to \cite{Frasch2013,Kouzoupis2016} to reduce the dimension of the coordination QP. 
Moreover, we prove that this coordination QP inherits the sparsity pattern from the original problem. 
We use this insight in the key step of our developments:
the introduction of a second (inner) level of problem distribution to ALADIN. In other words, we show how the coordination QP can be solved efficiently in a decentralized fashion. 
To the latter end, we propose a decentralized variant of the Conjugate Gradient (CG) method. 
We also investigate the application of decentralized ADMM. 
The proposed  framework is based on two consecutive layers of problem distribution: the general outer ALADIN structure is combined with a second inner layer. Hence we refer to it as \emph{bi-level distribution}.
As iterative methods (such as CG and ADMM) typically return inexact solutions, the original local convergence analysis for ALADIN \cite{Houska2016} is not directly applicable.
Accounting for this fact, we show that local convergence properties of ALADIN are preserved by enforcing bounds on the accuracy of the inner decentralized methods.
These bounds are derived using arguments from inexact Newton methods \cite{Dembo1982}.
This way we obtain---to the best of our knowledge---one of the first \emph{decentralized} algorithms with local convergence guarantees for constrained non-convex problems. 


The remainder is structured as follows: Section \ref{sec:Prelim} recalls ALADIN and condensing techniques for the coordination QP. 
Section \ref{sec:convAnal} shows how the local convergence properties of ALADIN while solving inexactly the coordination QP.
Section \ref{sec:distrSol} provides details on how to solve the reduced system in a decentralized fashion using decentralized ADMM \cite{Gabay1976,Boyd2011} and decentralized conjugate gradient.
Finally,  examples from power systems and from robotics are presented in Section \ref{sec:NumRes}.

\emph{Notation.}
If not explicitly stated differently, we use superscripts $(\cdot)^k$ for inner iterations  and omit outer iteration indexes for simplicity. In optimization problems the Lagrange multiplier $\kappa$ associated to constraint $h$ is denoted as $h(x)\leq 0 \;|\; \kappa$. Given a matrix $S\in \mathbb{R}^{n\times m}$, $S_{ij}$ denotes its $ij$th entry.

\section{Preliminaries \& Problem Statement } \label{sec:Prelim}

\subsection{Recalling ALADIN}
Distributed optimization aims at solving problems of the form\footnote{Ovserve that \eqref{eq:sepProb} can be interpreted as a generalization of a consensus problem \cite{Boyd2011} in the sense that any consensus problem can be expressed in form of \eqref{eq:sepProb} by appropriately choosing $A_i$.}${}^,$\footnote{
\label{footnote:constr}
For the sake of simplified notation we consider only inequality constraints $h_i$ here. Including equality constraints $g_i:\mathbb{R}^{n_{x_i}}\mapsto  \mathbb{R}^{n_{g_i}}$ does not pose any difficulty as $g_i$ can be reformulated as $0\leq g_i(x_i)\leq 0$.}
\begin{subequations}\label{eq:sepProb}
\begin{align} 
\min_{x\in\mathbb{R}^{n_x}} \,& \sum_{i\in \mathcal{R}} f_i(x_i) \\  
\text{subject to} \qquad h_i(x_i)&\leq 0 \quad \forall\, i \in \mathcal{R} \label{eq:IneqConstr}\\ 
\sum_{i\in \mathcal{R}}A_i x_i&=0, \label{eq:ConsConstr}
\end{align}
\end{subequations}
with objective functions $f_i:\mathbb{R}^{n_{x_i}}\rightarrow\mathbb{R}$ and constraints $h_i:\mathbb{R}^{n_{x_i}}\rightarrow\mathbb{R}^{n_{h_i}}$. In  all subproblems $i\in \mathcal{R}=\{1,\dots,N\}$ the functions $f_i$ and $h_i$ are assumed to be twice continuously differentiable  and possibly non-convex. The overall decision vector is 
$x:={(x_1^\top,\dots,x^\top_{N})}^\top \in \mathbb{R}^{n_x}$ and the matrices $A_i \in \mathbb{R}^{n_c\times n_{x_i}}$ describe couplings between subproblems.
 Standard ALADIN is summarized in Algorithm~\ref{alg:ALADIN2};  we refer to \cite{Houska2016,Engelmann2019,Jiang2017} for details including convergence proofs and application examples. 
\begin{algorithm}[t]
	\caption{Standard ALADIN (full-step variant)}
	\small
	\textbf{Initialization:} Initial guess $(z^0,\lambda^0)$, parameters $\Sigma_i\succ 0,\rho,\mu$. \\
	\textbf{Repeat} (until convergence):
	\begin{enumerate}
		\item \textit{Parallelizable Step:} Solve for all $i\in \mathcal{R}$ locally
		\begin{equation}
		\begin{aligned} \label{eq:locStep}
		x_i^k=\underset{x_i}{\text{arg}\text{min}}&\quad f_i(x_i) + (\lambda^k)^\top A_i x_i + \frac{\rho}{2}\left\|x_i-z_i^k\right\|_{\Sigma_i}^2 \\ 
		&\text{s.t.}\quad h_i(x_i) \leq 0 \quad \mid \kappa_i^k,		
		\end{aligned}
		\end{equation}
		and compute sensitivities $H_i^k$, $g_i^k$ and  $C_i^k$, cf. \cite{Houska2016}.
		\item \textit{Coordination Step:} Solve the coordination  problem \label{step:QPstep}
		\begin{align} \label{eq:globStep}
		\notag
		&\underset{\Delta x,s}{\text{min}}\;\sum_{i\in \mathcal{R}} \displaystyle \frac{1}{2}\Delta x_i^\top H^k_i\Delta x_i + {g_i^k}^\top \Delta x_i \hspace{-0.1em}  + \hspace{-0.1em} {\lambda^k}^\top  \hspace{-0.2em} s + \hspace{-0.1em}  \textstyle \frac{\mu}{2}\|s\|^2_2   \\ 
		&\begin{aligned} 
		\text{s.t.}                                &\;\;\, \sum_{i\in \mathcal{R}}A_i(x^k_i+\Delta x_i) =  s     &&|\, \lambda^\text{QP},\\
		&\;\;\;  C^{\mathrm{act}\,k}_i \Delta x_i = 0                                    &&\forall i\in \mathcal{R}.
		\end{aligned}
		\end{align}
		\item \textit{Broadcast and Update Variables:} \label{step:update}		 \noindent \\ 
 		$ \qquad 
		z^{k+1} \leftarrow  x^k + \Delta x^k\quad\text{and} \quad \lambda^{k+1}  \leftarrow  \lambda^\mathrm{QP}.
		$
	\end{enumerate}
	\label{alg:ALADIN2}
\end{algorithm}

Two main steps in ALADIN require central coordination and thus render ALADIN distributed instead of decentralized: (i) the coordination QP in Step~\ref{step:QPstep}) and  (ii) an additional globalization strategy which is neglected (for the sake of simplicity) in Step~\ref{step:update}). Here, we focus on designing a \emph{local} optimization algorithm. Hence, we use the full-step variant of ALADIN and  focus on issue (i). Note that---upon solving Step~\ref{step:QPstep}) exactly in a decentralized fashion and modulo technical subtleties---one directly obtains a decentralized algorithm for constrained non-convex problems \eqref{eq:sepProb}.

\subsection{Condensing the coordination QP}
In ALADIN  (Algorithm \ref{alg:ALADIN2}) the coordination QP \eqref{eq:globStep} directly scales with the number of decision variables and constraints $(n_x + n_h + n_c)$, which may be prohibitive in many applications. 
Hence we aim at reducing the size of \eqref{eq:globStep} to the number of coupling constraints $n_c$ which is typically much smaller than $(n_x + n_h)$. 
In context of direct methods for numerical optimal control, a similar approach has been used in \cite{Kouzoupis2016}. 
Subsequently, we derive the reduced QP based on the Schur-complement whereas the analysis in \cite{Kouzoupis2016} is based on dualization. In contrast to \cite{Kouzoupis2016}, we consider slack variables $s$  as they are important in practice.

In Step 2) of ALADIN one solves the coordination QP
\begin{align}
\label{eq::conqp}\notag
&\min_{\Delta x,s}& \sum_{i\in \mathcal{R}} \frac{1}{2}\Delta x_i^\top H_i\Delta x_i + &{g_i}^\top \Delta x_i
+ \lambda^\top s + \frac{\mu}{2}\|s\|^2_2\\
&\quad\text{s.t.} & C^\mathrm{act}_i \Delta x_i &= 0,  \qquad  \; \,\forall i\in \mathcal{R}\\ \notag
&&\sum_{i\in \mathcal{R}}A_i(x_i+\Delta x_i) &=  s \quad|\; \lambda^\mathrm{QP},
\end{align}
where $H_i \in \mathbb{R}^{n_{x_i} \times n_{x_i}}$ are positive definite Hessian approximations of the local Lagrangians, $g_i\in \mathbb{R}^{n_{x_i}}$ and the gradients, $\lambda \in \mathbb{R}^{n_c}$ are Lagrange multiplier estimates for the consensus constraint, $C^\mathrm{act}_i\in \mathbb{R}^{n^k_{h_i}\times n_{x_i}}$ are constraint linearizations of the active constraints with $n^k_{h_i}$ being the number of active constraints in the $k$th ALADIN iteration in subproblem $i \in \mathcal{R}$. $A_i \in \mathbb{R}^{n_c \times n_{x_i}}$ describes linear coupling between the subproblems. 
The slack $s \in \mathbb{R}^{n_c}$ in combination with a sufficiently large penalty parameter $\mu \in \mathbb{R}_{+}$ fosters numerical stability.\footnote{Neglecting the slack variables $s$ simplifies condensing. However, these variables are essential for handling inconsistent constraint linearizations \cite{Houska2016}. The examples of Section  \ref{sec:NumRes} fail to converge in absence of them.} For the sake of readability, we suppress the ALADIN iteration superscripts $(\cdot)^k$ whenever possible without ambiguity. 

\begin{assumption}[Strong regularity] \label{ass:LICQ}
For all ALADIN iterates $k \in \mathbb{N}$, for all $i\in \mathcal{R}$, and for all local minimizers of \eqref{eq:sepProb}, linear independence constraint qualification (LICQ), strict complementarity condition (SCC) and second-order sufficient conditions (SOSC) are satisfied on the nullspace of $C_i$, cf. \Cite{Nocedal2006}.\footnote{Note that the SOSC assumption made here is slightly stronger than the general SOSC from \cite{Nocedal2006}. Here we require positive definiteness of $H_i$ on the tangent space of the nonlinear constraints and do not consider the nullspace of the consensus constraints \eqref{eq:ConsConstr}.}
\end{assumption}
We employ the nullspace method  \cite{Nocedal2006} to project \eqref{eq::conqp} onto the subspace spanned by $C_i^\mathrm{act}$. 
Assumption~\ref{ass:LICQ} implies that $C^\mathrm{act}:=\operatorname{diag}_{i\in \mathcal{R}} C_i^\mathrm{act} \in \mathbb{R}^{n^k_{h}\times n_x}$ has full row-rank. Hence parametrizing $\operatorname{null}(C^\mathrm{act})$ in terms of $v\in \mathbb{R}^{n_x-n^k_{h}}$ yields  
\[
\operatorname{null}(C^\mathrm{act})=\{ x \in \mathbb{R}^{n_x} \; | \; x=Zv,\; v \in \mathbb{R}^{n_x-n^k_{h}}\}
\]
 where the columns of $Z \in \mathbb{R}^{n_x \times (n_x-n^k_{h})}$ form a basis of $\operatorname{null}(C^\mathrm{act})$. With $x_i:=Z_iv_i$ for all $i \in \mathcal{R}$, we write \eqref{eq::conqp} as
\begin{align} \notag 
&\min_{\Delta v,s}&& \sum_{i\in \mathcal{R}} \frac{1}{2}\Delta v_i^\top\bar H_i\Delta v_i +  {\bar g_i}^{\top} \Delta v_i + \lambda^{\top} s + \frac{\mu}{2}\|s\|^2_2\\  
\label{eq::redProb}
&\quad\text{s.t.} && \sum_{i\in \mathcal{R}}\bar A_i(v_i+\Delta v_i) =  s \quad|\; \lambda^\mathrm{QP}.
\end{align}
where $\bar H_i := Z_i^\top H_i Z_i$, $\bar g_i := Z_i g_i$,  and $\bar A_i := A_i Z_i$.
Upon eliminating the equation for $s$, the KKT conditions of \eqref{eq::redProb} read
\begin{equation} \label{eq:KKTsystem}
\begin{pmatrix}
\bar H & \bar A^\top \\
\bar A & -\frac{1}{\mu} I
\end{pmatrix}   
\begin{pmatrix}
\Delta v \\ \lambda^{\text{QP}}
\end{pmatrix}      
= 
\begin{pmatrix}
-\bar g \\ - \bar Av - \frac{1}{\mu} \lambda
\end{pmatrix},
\end{equation}
where $\bar H := \operatorname{diag}\left (\{\bar H_i\}_{i \in \mathcal{R}}\right )$, $\bar g := \left (\,\bar g_1^\top,\; \dots,\; \bar g_N^\top\, \right )^\top$ and  ${\bar A:= \left (\,\bar A_1,\; \dots\;, \bar A_N \,\right )}$.

We use the \emph{Schur-complement} \cite[Chap. 16]{Nocedal2006} to further reduce \eqref{eq:KKTsystem}.  SOSC implies that $\bar H$ is positive definite and therefore invertible. Hence, we solve the first row of \eqref{eq:KKTsystem} as $\Delta v = \bar H^{-1}(-{\bar A}^\top \lambda^{\text{QP}} -\bar g)$ and obtain
\begin{equation} \label{eq:redRedSystem}
(  \mu^{-1}I + \bar A \bar H ^{-1} \bar A^\top) \lambda^{\text{QP}} = \bar A (v  - \bar H ^{-1}\bar g) + \mu^{-1} \lambda
\end{equation}
which is a linear system of equations of dimension ${n_c}$. 
After solving \eqref{eq:redRedSystem}, the solution to \eqref{eq::conqp} $\Delta x$ can be obtained by backwards substitution.
Exploiting the block structure of $\bar H$, $\bar g$ and $\bar A$ we write \eqref{eq:redRedSystem} as
\begin{equation} \label{eq:redRedSystem2}
\left(\mu^{-1}I + \sum_{i\in \mathcal{R}} S_i  \right) \lambda^{\text{QP}} = \mu^{-1} \lambda + \sum_{i\in \mathcal{R}} s_i  
\end{equation}
where $S_i=\bar A_i \bar H_i^{-1} \bar A_i^\top$ and $s_i=\bar A_i (v_i -  \bar H_i ^{-1} \bar g_i)$. Observe that the matrices $S_i$ and the vectors $s_i$ can be computed entirely locally. 
	Furthermore, reverse application of the above formulas shows that the increments $\Delta x_i$ can be computed locally via 
	\begin{equation} \label{eq:backSub}
	\Delta x_i = Z_i \bar H_i^{-1}\left ( -\bar A_i^\top \lambda^{\text{QP}}-\bar g_i\right ).
	\end{equation}
Doing so, we arrive at a variant of ALADIN requiring  less communication compared to the standard one in   Algorithm~\ref{alg:ALADIN2}.


\subsection{Bi-level distributed ALADIN}

\begin{algorithm}[t]
	\caption{Bi-level distributed ALADIN}
	\small
	\textbf{Initialization:} Initial guess $(z^0,\lambda^0)$, parameters $\Sigma_i\succ 0,\rho,\mu$. \\
	\textbf{Repeat} (until convergence):
	\begin{enumerate}
		\item \textit{Parallelizable Step:} Solve for all $i\in \mathcal{R}$ locally \label{step:LocStepBil}
		\begin{equation}
		\begin{aligned} \label{eq:locStepbil}
		x_i^k=\underset{x_i}{\text{arg}\text{min}}&\quad f_i(x_i) + (\lambda^k)^\top A_i x_i + \frac{\rho}{2}\left\|x_i-z_i^k\right\|_{\Sigma_i}^2 \\ 
		&\text{s.t.}\quad h_i(x_i) \leq 0 \quad \mid \kappa_i^k,	
		\end{aligned}
		\end{equation}
		and compute \emph{condensed}  sensitivities $\tilde S_i$ and $\tilde s_i$.
		\item \textit{Coordination Step:} Solve decentralized/distributed \label{step:GlobStepBil}
		\begin{equation} \label{eq:redRedSystemAlg}
 		\left(\mu^{-1}I + \sum_{i\in \mathcal{R}} S_i  \right) \lambda^{\text{QP}} = \mu^{-1} \lambda + \sum_{i\in \mathcal{R}} s_i 
		\end{equation}
		with residuum $		\|r_\lambda^k\|$ \eqref{eq:lamResidual} small enough according to \eqref{eq:pRes}.
		\item \textit{Broadcast and 
		Update Variables:} \\
		$\lambda^k\leftarrow\lambda^{\text{QP}}$ and $z_i^{k+1}\leftarrow x_i^k + \Delta x_i^k$ using \eqref{eq:backSub}.
	\end{enumerate}
	\label{alg:ALADINbil}
\end{algorithm}

Algorithm \ref{alg:ALADINbil} summarizes the general algorithmic framework for bi-level distributed ALADIN. 
Note that the condensing all iterates $k$ for \eqref{eq::conqp} can be performed locally and that a coordination QP of reduced dimension is used for coordination. This distributed algorithm can in principle be applied as is. However, it still requires solving a (less complex) hierarchical coordination problem \eqref{eq:redRedSystemAlg}. 
Observe that solving \eqref{eq:redRedSystemAlg} by a decentralized algorithm, one obtains a decentralized variant of ALADIN. In Section \ref{sec:distrSol} we propose two variants for doing so: one based on conjugate gradient and one based on ADMM. 
As, for conceptual and numerical reasons, these iterative algorithms do not yield an exact values of $\lambda^{\text{QP}}$, the next section presents a convergence analysis of ALADIN for inexact solutions to  \eqref{eq:redRedSystemAlg}.

\section{Local Convergence Analysis} \label{sec:convAnal}
Usually decentralized algorithms solving \eqref{eq:redRedSystemAlg} achieve a finite precision only. 
Hence it is fair to ask whether  it is possible to preserve  local convergence guarantees under inexact solutions. 
We answer this question by combining properties of ALADIN \cite{Houska2016} with classical results from inexact Newton methods \cite{Dembo1982}.

Bi-level distributed ALADIN is composed of two main steps: the parallelizable Step~\ref{step:LocStepBil}) solving local NLPs and computing (condensed) sensitivities as well as the coordination Step~\ref{step:GlobStepBil}) solving \eqref{eq:redRedSystemAlg}. In order to establish local convergence properties of bi-level distributed ALADIN, we aim at ensuring progress towards a local minimizer in both steps.

From \cite[Lemma 3]{Houska2018} we have that the mapping formed by Step~\ref{step:LocStepBil}) is locally Lipschitz, i.e.
\begin{equation} \label{eq:minLip}
\|p^{k}-p^\star\| \leq \chi \|q^{k}-p^\star\|
\end{equation}
with $q^k=(z^k, \lambda^k,\kappa^{k-1})$ and $p^k:=(x^k,\lambda^k,\kappa^k)$ for some $\chi < \infty$.
The superscript $(\cdot)^\star$ denotes optimal primal and dual variables of 
\eqref{eq:sepProb}.

It remains to analyze the progress in the coordination problem~\eqref{eq:redRedSystemAlg}.
Eliminating $s$, the optimality conditions of \eqref{eq:globStep} read
\begin{equation} \label{eq:redKKT-QPdelta}
\underbrace{\begin{pmatrix} 
H  & A^\top & C^{a^\top}  \\
A & -\frac{1}{\mu}I & 0 \\
C^{a}   & 0 & 0 
\end{pmatrix}}
_
{=:M(p^k)}
\Delta q^k \hspace{-1mm}
= \hspace{-1mm}
\underbrace{\begin{pmatrix}
-g -C^{a\top} \kappa^k -A^\top \lambda^k \\ -Ax^k +b \\ 0
\end{pmatrix}}
_
{=:m(p^k)},
\end{equation}
with $\Delta q^k=q^{k+1}-p^k$.
Apart from the entry $-\frac{1}{\mu}I$, \eqref{eq:redKKT-QPdelta} is equivalent to a Newton step for \eqref{eq:sepProb} if exact Hessians and Jacobians are used.
Hence, we have the typical progress in Step~\ref{step:QPstep}) known from Newton-type methods  \cite{Nocedal2006}
\[
\|q^{k+1}-p^\star\|\leq \gamma \|p^k-p^\star\| + \frac{\omega}{2}\|p^k-p^\star\|^2_2
\]
where $\gamma=\|I-M(p^k)^{-1}\nabla^2\mathcal{L}(p^k)\|<1$ can be seen as a bound on the error of $\nabla^2\mathcal{L}(p^k)$  with $\mathcal{L}(x,\lambda,\kappa):=f(x)+\lambda^\top Ax + \kappa^\top h(x)$ being the Lagrangian to \eqref{eq:sepProb}.
Yet this only holds for an exact solution to \eqref{eq:redKKT-QPdelta}. 

Denote the approximate solution  by $\bar q^{k+1}$ and $ \Delta \bar q^k= \bar q^{k+1}-p^k$.
We define the residual for \eqref{eq:redKKT-QPdelta} similar to inexact Newton methods as
\[
r_p^k:= M(p^k)\Delta \bar q^k - m(p^k).
\]
We assume that the residual is bounded by 
\begin{equation} \label{eq:pRes}
\|r_p^k\|\leq \eta^k\|m(p^k)\|,
\end{equation} 
which we have to guarantee during the ALADIN iterations.
Now we have all the ingredients to prove the main result of this section.
\begin{theorem}[Conv. of Bi-level decentralized ALADIN]~\\
	Consider bi-level distributed  ALADIN  (Algorithm \ref{alg:ALADINbil}).
	Suppose Assumption~\ref{ass:LICQ} holds. Let $\frac{1}{\mu^k}=O(\|q^k-p^\star\|)$, let $\nabla^2\mathcal{L}$ and $\nabla\mathcal{L}$ be Lipschitz, and let the solution to the condensed QP \eqref{eq:redRedSystemAlg} satisfy \eqref{eq:pRes} in each iteration $k\in \mathbb{N}_+$. 
	
	Then there exists  $\eta\in (\eta^k,\,\infty)$ such that bi-level distributed ALADIN converges locally to $(x^\star,\lambda^\star,\kappa^\star)$ 
	\begin{itemize}
		\item at linear rate; and
		\item at quadratic rate if $\eta^k=O(\|q^{k}-p^\star\|)$.
	\end{itemize}
\end{theorem}
\begin{proof}
The inequalities \eqref{eq:minLip}, \eqref{eq:pRes} and the Lipschitz property of $m$ with  $\frac{1}{\mu^k}=O(\|q^k-p^\star\|)$ imply
\begin{align*}
\|\bar q^{k+1}-p^\star\|&\leq \|q^{k+1}-p^\star\|+\|\bar q^{k+1} - q^{k+1}\| \\
&\leq \frac{\omega}{2}\|q^k-p^\star\|^2_2 + \alpha\cdot \eta^k\|m(p^k)-m(p^\star)\| \\
&\leq \frac{\omega}{2}\|q^k-p^\star\|^2_2 + \alpha\cdot \beta\cdot \eta^k\|p^k-p^\star\|\\
&\leq \frac{\omega}{2} \|q^k-p^\star\|^2_2 + \alpha\cdot \beta\cdot\chi \cdot \eta^k\|q^k-p^\star\|,
\end{align*}
where $\beta$ is the Lipschitz-constant of $m$.
The finiteness of $\alpha,\beta$ and $\chi$ shows linear convergence if ${\alpha\cdot\beta\cdot\chi\cdot\eta^k<1}$.
Quadratic convergence follows immediately from the above inequality if $\eta^k=O(\|q^k-p^\star\|)$.
\end{proof}
The above result shows that inexact solutions to \eqref{eq:redRedSystemAlg} do not jeopardize linear or even quadratic local convergence of bi-level distributed ALADIN. 

However, the question of  how to evaluate \eqref{eq:pRes} in a decentralized setting arises. 
To this end, we draw upon $r_p^k$ the residual of \eqref{eq:redRedSystemAlg}
\begin{align} \label{eq:lamResidual}
r_\lambda^k:=\left(\mu^{-1}I + \sum_{i\in \mathcal{R}} S_i  \right) \lambda^k - \mu^{-1} \lambda - \sum_{i\in \mathcal{R}} s_i.
\end{align}
The structure of $ S$, $s$, and \eqref{eq:lamResidual} imply that 
$r_\lambda^k = \bar A\bar H^{-1}(-\bar g-\bar A^\top \lambda^k) = \bar A \Delta v= \nabla_{\lambda} \mathcal{L}(p^k)$.
As we enforce $\nabla_{\Delta x} \mathcal{L}(p^k)=0$ and $\nabla_{\kappa} \mathcal{L}(p^k)=0$ by virtue of the nullspace method and the first row of \eqref{eq:KKTsystem}, we obtain $r_p^k={(0^\top \ \ 0^\top \ \ r_\lambda^{k\top})}^\top$ and $\|r_p^k\|=\|r_\lambda^k\|$. 
Hence, note that one can evaluate \eqref{eq:pRes} using only the residual of the reduced system $\|r^k_\lambda\|$.

\section{Decentralized  Solution of the Coord. QP \eqref{eq:redRedSystemAlg}} 
\label{sec:distrSol}
Observe that the QP \eqref{eq:redRedSystemAlg} inherits structural properties of problem \eqref{eq:sepProb}; i.e. the  Schur-complements $S_i$ inherit the sparsity pattern induced by the coupling matrices $A_i$. 
This sparsity can be exploited---either to further reduce communication by using sparse matrix storage formats or to design decentralized algorithms.
Here we focus on the latter. We first analyze the sparsity of the matrices $S_i$s and then we propose two decentralized algorithms exploiting this sparsity. 

\subsection{Sparsity of the Schur-complements}

Usually, the consensus constraint \eqref{eq:ConsConstr} describes  couplings between two neighboring subproblems $i, j \in \mathcal{R}$. This means that in the matrices $A_i$ and $A_j$  the $i$th and $j$th rows are nonzero. 
\begin{definition}[Assigned consensus constraints] \label{def:nonAss}
	A subproblem $i\in \mathcal{R}$ is called assigned to consensus constraint $j \in{ \mathcal{C} = \{1,\dots,n_c\}}$, if  the $j$th row of $A_i$ is non-zero.
	Furthermore, all subproblems assigned to consensus constraint $j$ are collected in $\mathcal{R}(j):=\{i\in \mathcal{R}\;|\; i \; \text{assigned to}\; {j \in \mathcal{C}}\}$. 
	
	A consensus constraint $j \in \mathcal{C}$ is called $n$-assigned, if $|\mathcal{R}(j)|=n.$ Furthermore, if $|\mathcal{R}(j)|\leq n$ for all $j\in \mathcal{C}$, problem \eqref{eq:sepProb} is called $n$-assigned.
\end{definition}
Observe that assigned consensus constraints generalize the usual consensus setting \cite{Boyd2011}. Moreover, they provide an  effective framework to analyze the sparsity pattern of the Schur-complements. 
We remark that any generic consensus problem can be expressed in this form via appropriate choice of $A_i$ and using additional local variables.
\begin{remark}[Reformulation as 2-assigned problem]
	Without loss of generality any $n$-assigned problem can be reformulated as $2$-assigned problem by introduction of auxiliary decision variables. 
	For example consider a $3$-assigned problem with consensus constraint
$	A_1x_1 + A_2 x_2 + A_3x_3=0$
	where $A_1,A_2,$$A_3 \neq 0$. 
	Introduce a copy of $x_2$ in subproblem $1$ as $\tilde x_2 := x_2$ and define an augmented decision vector $\tilde x_1:=(x_1 \; \tilde x_2)^\top$. This yields a $2$-assigned problem in terms of the augmented decision vectors $(\tilde x_1, x_2, x_3)$
	\begin{align*}
	(A_1 \quad A_2)\, \tilde x_1 + A_3x_3 = 0, \quad
	(0\quad I) \, \tilde x_1 - I x_2 = 0.
	\end{align*}
\end{remark}

\begin{lemma}[Sparsity of $S_i$] \label{lem:sparsity}
	The rows and columns of $S_i$ and entries of $s_i$, $i \in \mathcal{R}$, which are not assigned to consensus constraint $j$, (i.e. all  $j \notin \mathcal{C}(i):= \{j \in \mathcal{C}\; |\; i \in \mathcal{R}(j) \}$) are zero.
\end{lemma}
\begin{proof}
	We have $S_i=  \bar A_i \bar H_i^{-1} \bar A_i^\top= A_i(Z_i \bar H_i^{-1} Z_i^\top) A_i^\top$. All columns of $A_i$ with $j \notin \mathcal{C}(i)$ are zero by Definition \ref{def:nonAss}. It follows immediately that the rows and columns of $S_i$ with $j \notin \mathcal{C}(i)$ are zero. The sparsity of  $s_i=\bar A_i (v_i -  \bar H_i ^{-1} \bar g_i)$ follows analogously.
\end{proof}
Lemma \ref{lem:sparsity} shows that the matrices $S_i$ and vectors $s_i$ have non-zero entries only for neighboring subproblems. 
\vspace*{-3mm}

\subsection{Consensus reformulation}

Now, we reformulate \eqref{eq:redRedSystemAlg} as a strictly convex consensus problem such that the conjugate gradient method and ADMM are applicable.
Specifically, we reformulate \eqref{eq:redRedSystemAlg} as
\begin{equation} \label{eq:sumKKTsystem}
\left (\sum_{i\in \mathcal{R}} \tilde S_i \right ) \lambda^{\text{QP}} = \sum_{i\in \mathcal{R}}  \tilde s_i
\end{equation}
where each $\tilde S_i$ and $\tilde s_i$ is constructed by local information only.
Equation \eqref{eq:redRedSystemAlg} implies that the reduced QP is separable as it involves sums of $S_i$ and $s_i$. 
However, the terms $\mu^{-1} I$ and $\mu^{-1}\lambda$ can not directly be assigned to any of the subproblems. 

One possibility is to introduce an additional subproblem which would serve as a coordinator. However, here we are interested in relying on neighborhood communication only. 
Hence we distribute $\mu^{-1} I$ and $\mu^{-1}\lambda$ uniformly to all subproblems assigned to the corresponding consensus constraint. This yields
\begin{equation} \label{eq:sparsDef}
\tilde S_i := S_i + \sum_{j=1}^{n_c} \frac{\delta_{ij}}{|\mathcal{R}(j)|\, \mu } I_j,  \; \tilde s_i := s_i + \sum_{j=1}^{n_c} \frac{\delta_{ij}}{|\mathcal{R}(j)|\, \mu } \lambda I_j,
\end{equation}
where $I_j$ contains only zeros except for $I_{jj}=1$ and $\delta_{ij} := 1$ if $j \in \mathcal{C}(i)$ and $0$ else. 
This way \eqref{eq:redRedSystemAlg} is expressed in the form of \eqref{eq:sumKKTsystem}  without destroying its sparsity pattern.

The next result reformulates \eqref{eq:sumKKTsystem} as strictly convex QP.
\begin{lemma}[Minimization to solve \eqref{eq:redRedSystem2}] \label{lem:SposDef}
	The minimizer of
	\begin{align} \label{eq:minLam}
	&\min_{\lambda}\; \frac{1}{2}\lambda^\top \sum_{i=1}^{N} \tilde  S_i  \lambda - \sum_{i=1}^{N} \tilde s_i^\top \lambda.
	\end{align}
	is unique and solves \eqref{eq:sumKKTsystem}. 
	Furthermore, \eqref{eq:minLam} is strictly convex.
\end{lemma}
\begin{proof}
	The first-order necessary condition for \eqref{eq:minLam} reads $\frac{1}{2}(\sum_{i=1}^{N} \tilde S_i + \sum_{i=1}^{N} \tilde S_i^\top)  \lambda - \sum_{i=1}^{N} \tilde s_i =0$. From Lemma \ref{lem:SposDef} (given in the Appendix) one has that $\tilde S_i= \tilde S_i^\top$. This proves the first assertion. Moreover,  Lemma \ref{lem:SposDef} gives $\sum_{i=1}^N \tilde S_i \succ 0$  which implies strict convexity of \eqref{eq:minLam}.
\end{proof}

\subsection{Decentralized conjugate gradient}
\begin{algorithm}[t]
	\small
	\caption{Decentralized conjugate gradient}
	\textbf{Initialization:} Initial guess $r^0=p^0= \tilde s-\tilde  S\lambda^0 $.\\
	\textbf{Preparation:} Exchange $\tilde S_{i}e_j$ and $\tilde s_i^\top e_j$  between neighboring regions $\mathcal{R}(j)$ locally for all $j \in \mathcal{C}$.\\
	\textbf{Repeat:}
	\begin{enumerate}
		\item \textit{Compute locally} 
		\[
		 \mathsf{r}^{\mathsf{S}k}_j= \sum_{j\in \mathcal{C}} r^{k\top} \tilde Se_j r^k_j  \text{ and } \mathsf{r}^{\mathsf{2}k}_j=(r_j^k)^2,
		 \]
		 between all $\mathcal{R}(j)$ for all $j \in \mathcal{C}$ using \eqref{eq:sparsityExploitS}.		
		\item \textit{Sum up globally} $\alpha^k=\left({\sum_{j\in \mathcal{C}} \mathsf{r}_j^{\mathsf{2}k}}\right)\Huge /\left({\sum_{j\in \mathcal{C}} \mathsf{r}_j^{\mathsf{S}k}}\right )$. \label{step:globSum1}
		\item \textit{Compute locally for all $j \in \mathcal{C}$} 
		\begin{align*}
		\lambda_j^{k+1} &= \lambda^k_j + \alpha^k p^k_j,\\
		r^{k+1}_j &= r^k_j-\alpha^k \sum_{i \in \mathcal{C}} \tilde S_{ji} p^{k}_i,\\
		\mathsf{r}_j^{\mathsf{2}k+1}&=(r_j^{k+1})^2.
		\end{align*}
		\item \textit{Sum up globally} $\beta^k=\dfrac{1}{\mathsf{r}_j^{\mathsf{2}k}}
		\displaystyle\sum_{j\in \mathcal{C}} \mathsf{r}_j^{\mathsf{2}k+1}$.
		\label{step:globSum2}
		\item \textit{Compute locally} $p_j^{k+1} = r^k_j + \beta^k p^k_j$ for all $j \in \mathcal{C}$.
	\end{enumerate}
	\label{alg:dCG}
\end{algorithm}

Next, we propose a sparsity exploiting variant of the conjugate gradient algorithm. 
The usual centralized conjugate gradient method with $r^0=p^0=\tilde s-\tilde S\lambda^0$ reads \cite{Nocedal2006}
\begin{subequations} \label{eq:CGiter}
	\begin{align}
	\alpha^k &= \frac{r^{k\top}r^k}{r^{k\top}\tilde Sr^k}, \label{eq:CGalpha}\\
	\lambda^{k+1} &= \lambda^k + \alpha^k p^k, \label{eq:CGlam}
	\end{align}
	\begin{align}
	r^{k+1} &= r^k-\alpha^k\tilde Sp^k, \label{eq:CGr}\\
	\beta^k &= \frac{r^{k+1 \top} r^{k+1}}{r^{k^\top}r^k}, \label{eq:CGbeta} \\
	p^{k+1} &= r^{k+1} + \beta^k p^k. \label{eq:CGp}
	\end{align}
\end{subequations}
 Recall that $\tilde S=\sum_{i \in \mathcal{R}}\tilde S_i$ and let $e_j$ be the $j$th unit vector. Then,  from Lemma \ref{lem:sparsity} and \eqref{eq:sparsDef} we have
\begin{align} \label{eq:sumSi}
\tilde S e_j = \sum_{i\in \mathcal{R}}\tilde S_i e_j = \sum_{i\in \mathcal{R}(j)}  \tilde S_i e_j,
\end{align}
i.e. the $j$th column of $\tilde S$ belonging to consensus constraint $j$ is the sum \emph{only} of the respective rows of $\tilde S_i$ of the subproblems $i\in \mathcal{R}(j)$ assigned to consensus constraint  $j \in \mathcal{C}$. 
Therefore the rows of $\tilde S$ can be constructed locally based on neighborhood communication between the assigned subproblems. 
Furthermore, in \eqref{eq:CGalpha} we have to compute 
\begin{align} \label{eq:rTopSi}
r^{k\top}\tilde S= \left (r^{k\top} \tilde Se_1,\;\dots\;,r^{k\top}\tilde S e_{n_c} \right ).
\end{align}
From \eqref{eq:sumSi} and Lemma \ref{lem:sparsity} we know that $\tilde S_{i}e_j=0$ for $i \in \mathcal{C}\setminus \cup_{i \in \mathcal{R}(j)} \mathcal{C}(i)$. Hence, the components of \eqref{eq:rTopSi} are
\begin{align} \label{eq:sparsityExploitS}
r^{k\top} \tilde Se_j =  \sum_{i \in \mathcal{C}} r^{k}_i \tilde S_{ij} = \hspace{-0.1em} \sum_{i \in \cup_{l \in \mathcal{R}(j)} \mathcal{C}(l)} \hspace{-0.1em} r^{k}_i \tilde S_{ij}, \quad \forall j \in \mathcal{C},
\end{align}
where $\tilde S_{ij}$ denotes the $ij$th element of $\tilde S$.
Observe that it suffices to exchange $r^{k}_i$ and $\tilde S_{ji}$ locally between all $i\in \mathcal{R}(j)$. As 
\begin{equation} \label{eq:summands_rSr}
r^{k\top}\tilde Sr^k = \sum_{j\in \mathcal{C}}\left ( r^{k\top} \tilde Se_j \right  )r^k_j =  \sum_{j\in \mathcal{C}} r^{k\top} \tilde Se_j r^k_j 
\end{equation}
and $r^k_j$ is also known locally, all summands in \eqref{eq:summands_rSr} can be computed locally. The only centralized operation is evaluating one global sum. The same applies to 
\begin{align*}
r^{k\top}r^k = \sum_{j\in \mathcal{C}} (r_j^k)^2,
\end{align*}
where $(r_j^k)^2$ can be computed locally. Similar analysis applies  to \eqref{eq:CGlam}-\eqref{eq:CGp}, where in \eqref{eq:CGbeta} an additional global sum  is needed and therefore the conjugate gradient needs two global sums in each iteration.\footnote{Note that although the sum is global, it can easily be decentralized by computing the sum via ``hopping" (i.e. a round-robin protocol) from neighbor to neighbor.} Algorithm \ref{alg:dCG} summarizes  the proposed decentralized variant of the conjugate gradient method.

Note that the  decentralized conjugate gradient algorithm requires communication between all subproblems assigned to a specific consensus constraint. In other words, this algorithm can be executed in decentralized fashion if the coupling described in the $A_i$s refer to two subproblems only, i.e. if Problem~\eqref{eq:sepProb} is \emph{$2$-assigned}. The same holds for ADMM as we will see in the next section. 

\subsection{Decentralized ADMM}
The above proposed decentralized conjugate gradient method still requires (very little) central coordination using the global sums in Step~\ref{step:globSum1}) and Step~\ref{step:globSum2})  of Algorithm \ref{alg:dCG}. As an alternative, we consider a decentralized variant of ADMM for solving \eqref{eq:redRedSystemAlg} without these centralized steps. 

We rely on decentralized ADMM in so-called consensus form to \eqref{eq:minLam}   \cite{Boyd2011,Bertsekas1989}.  To this end, we introduce variable copies of $\lambda$, $\lambda_1,\dots,\lambda_N$ and  write \eqref{eq:minLam} as
\begin{equation} \label{eq:consProb}
\begin{aligned}
\min_{\lambda_1,\dots,\lambda_N,\bar  \lambda} \quad&\sum_{i=1}^{N} f_i(\lambda_i) \\
\;\text{s.t.} \quad&\lambda_i=\bar \lambda\;\; | \;\; \gamma_i, \qquad i=1  ,\dots,N,
\end{aligned}
\end{equation}
with $f_i(\lambda_i) := \lambda_i^\top S_i \lambda_i - s_i^\top \lambda_i$. 
The ADMM iteration rules can be derived from the method of multipliers combined with coordinate descent \cite{Bertsekas1989}. 
Decentralized ADMM is summarized in Algorithm~\ref{alg:dADM}.
Observe that \eqref{eq:lamUp} is an entirely local step, \eqref{eq:lamBarUp}  is a simple averaging step based on neighborhood communication, and \eqref{eq:gamUp} is again a local step. 
Furthermore \eqref{eq:lamUp} requires solving a linear system with changing right-hand sides, which means that $(\tilde S_i + \rho I)$ has to be factorized once only and can be reused in all ADMM iterations. 

	\begin{algorithm}[b]
			\small
		\caption{Decentralized ADMM}
		\textbf{Initialization:} Initial guess  $\lambda^0$ and parameter $\rho=\rho^\text{ADM}$.\\
		\textbf{Repeat:}
		\begin{enumerate}
			\item \textit{Compute locally} for all $i\in \mathcal{R}$
			\begin{equation} \label{eq:lamUp}
			\lambda_i^{k+1} \hspace{-1mm}= \hspace{-0.5mm} \underset{{\lambda_i}}{\operatorname{argmin}} \; \lambda_i^\top \hspace{-1mm} \left  ( \tilde S_i +\rho I \right) \hspace{-0.5mm}  \lambda_i  
			+ \left (\gamma^k_i - \tilde s_i - \rho \bar \lambda^k \right )^\top \hspace{-1.5mm} \lambda_i.
			\end{equation}
			
			\item \textit{Compute locally} for all consensus constraints $j\in \mathcal{C}$
			\begin{align} \label{eq:lamBarUp}
			e_j^\top \bar \lambda^{k+1} &=\frac{1}{|\mathcal{R}(j)|} \sum_{i \in \mathcal{R}(j)} e_j^\top \lambda_i^{k+1}.
			\end{align}
			\item \textit{Compute 	locally} for all $ i \in \mathcal{R}$
			\begin{align} \label{eq:gamUp}
			\gamma_i^{k+1} = \gamma_i^k + \rho\left (\lambda_i^{k+1}-\bar \lambda_i^{k+1}\right ).
			\end{align} \vspace*{-3mm}
		\end{enumerate}
		\label{alg:dADM}
\end{algorithm}

\subsection{Comparison of CG and ADMM}
The convergence properties of CG and ADMM are summarized in Table \ref{tab::convPropCGAL}. In theory, CG yields the \emph{exact} solution in at most $n_c$ steps \cite[Thm 5.1]{Nocedal2006}. 
However, in practice the convergence is typically slower as conjugate gradient is sensitive to errors caused by finite precision arithmetic.  Practically one  observes superlinear convergence \cite{Beckermann2001}.
The recent paper \cite{Makhdoumi2017}  shows sublinear convergence of ADMM for convex objectives $f_i$.\footnote{For strongly convex $f_i$, linear convergence of ADMM can be shown \cite{Yang2016,Nedic2018a,Makhdoumi2017}. In the present paper  the $f_i$ of \eqref{eq:consProb} are only convex but not strictly convex.} 
In case of \eqref{alg:ALADINbil}, the $f_i$s are only convex, hence at least sublinear convergence can be expected which is in line with our later numerical observations.
Thus conjugate gradient is expected to outperform ADMM.
An advantage of CG compared to ADMM is that no tuning of the step size is needed, as this is done ``automatically'' in Step~2) and Step~4) of CG.

As discussed in the previous section, satisfying \eqref{eq:pRes} preserves the convergence properties in  bi-level distributed ALADIN.
Note that  criterion \eqref{eq:pRes} can be evaluated locally by computing $e_j^\top r^k_\lambda$ for each $j \in \mathcal{C}$ and calculating one additional global sum. 
 However, in implementations it turns out that   a fixed number of iterations for the coordination step combined with warm starting  often suffices to ensure $0<\eta_k<0$. 

\begin{table}[h]
	\centering
	\renewcommand{\arraystretch}{1.2}
	\caption{Convergence properties of decentralized CG and decentralized ADMM for \eqref{eq:redRedSystemAlg}.}
	\begin{tabular}{rlll}
		\toprule
		conv. rate		&	CG & ADMM  \\
		\midrule
		theoretical&  $n_c$-step & sublinear$^{\footnotesize 8}$   \\
		practical &  linear/superlinear\footnote{Analyzing the convergence rate of conjugate gradient methods seems quite complex as there are different phases with different convergence rates during the iteration cf. \cite{Axelsson2014,Beckermann2001}. However, the practically observed convergence rate often is superlinear \cite{Axelsson2014}.}  & sublinear \\
		tuning & no & yes \\
		\bottomrule
	\end{tabular}
	\label{tab::convPropCGAL}
\end{table}
\begin{remark}[Related works on optimization over networks]
	Related results to  our above developments can be found in the context of distributed optimization over networks, see \cite{Nedic2018,Nedic2018a} for recent overviews. The problems considered therein are in general more difficult.  
	Frequently,  communication delays, a time-varying network topology and asynchronous operation might be considered.  
	Prominent algorithms tailored to distributed optimization over networks are, for example, EXTRA \cite{Shi2015}, NEXT and also the widely used decentralized variant of ADMM \cite{Shi2014}. Linear systems of equations are considered in \cite{Lu2018b,Mou2015}, gradient and subgradient-based algorithms can be found in \cite{Nedic2010,Jakovetic2014}. 
		Indeed most of the algorithms cited above can in principle be used to solve \eqref{eq:minLam} in decentralized fashion. 	A potential pitfall might be that the convergence rate of these algorithms is at most linear, in many cases merely  sublinear.
%
\end{remark}

\subsection{Communication analysis} \label{se:commAnal}
We turn to analyze the forward communication  need in all ALADIN variants for 2-assigned problems. 
Forward means that, for the sake of  simplicity, we consider the  communication in Step~2) of  the  different ALADIN variants where local sensitivities are communicated to the coordination QP. The backward communication in Step~3) is negligible compared to forward one. Our analysis evaluates communication by counting the  number of floating point numbers.

Moreover, we distinguish two different kinds of communication: The first one is global communication, i.e. the information sent to any central (coordinating) entity. 
The second kind is local communication between neighbors. 
We assess the local preparation steps, which are done only once per outer ALADIN iteration  in a preprocessing phase between neighboring subproblems.\footnote{Note that we analyze the communication under symmetric conditions; i.e. both regions assigned to a consensus constraint send and receive the values corresponding to the respective consensus constraint. In general, it would suffice to choose one of these two participating regions to take care of the computations. However, this would render the algorithm somehow asymmetric.}


\begin{table}[t]
	\centering
	\caption{Per-step forward communication (number of floats) for 2-assigned problems and different ALADIN variants.	}
	\begin{tabular}{rcccc}
		\toprule
		 variant	&	standard & cond.  &  ADMM &  CG    
		\\ \midrule 		
		local prep.&	-	  & -  & -&  $2 n_c^2$  $\phantom{\displaystyle \sum}$   \\
		local iter.   & - & - & $2n_c n^{\mathrm{AD}}$ &$ 2n_c n^{\mathrm{CG}}$ $\phantom{\displaystyle \sum^N}$\\
		global& $>\hspace{-0.3em}\displaystyle \sum_{i=1}^N\frac{(n_{x_i}+n_{g_i})^2}{2}$ & ${{n_c^2 + n_c}}$ & - & $2N n^{\mathrm{CG}}$ $\phantom{\displaystyle \sum_{i=1}^N}$
		\\ \bottomrule
	\end{tabular}
	\label{tab:forwComm}
\end{table}

The forward communication for solving the coordination problem~\eqref{eq:redRedSystemAlg} of bi-level distributed ALADIN once is shown in Table \ref{tab:forwComm}.
In its full variant, ALADIN communicates the first and second-order sensitivities of the objective and the first-order sensitivity of the constraints to the coordinator.  
Let the constraints $h_i$ \eqref{eq:IneqConstr} consist of $n_{gi}$ equalities (handled as per Footnote \ref{footnote:constr}) and $n_{h_i} - n_{g_i}$ inequalities.
Neglecting sparsity and counting the number of all entries of the sensitivity matrices/vectors yields the following lower bound
$
\sum_{i=1}^N\frac{(n_{x_i}+n_{g_i})(n_{x_i}+n_{g_i}+1)}{2}.
$
Note that we do not count the communication of the $A_i$s here as they have to be communication only once and do not change during iterations.
In case of active inequality constraints, $n_{g_i}$ is enlarged by the number of active inequality constraints which is bounded by $n_{h_i} - n_{g_i}$. 
Hence, the above is a lower bound on the per-step communication which may vary during the ALADIN outer iterations. 
For a detailed application-specific communication analysis for the standard ALADIN see  \cite{Engelmann2019}.

In the condensed and sparsity exploiting variant of ALADIN---i.e. Algorithm \ref{alg:ALADINbil} without decentralization of \eqref{eq:redRedSystemAlg}---the global forward communication is 
$
{n_c(n_c+1)}
$ 
where $n_c$ is the number of coupling constraints. The number of coupling constraints is typically much smaller than the total number of decision variables thus reducing the necessary communication effectively. 
Note that the $2$ in the denominator disappears due to 2-assignment and therefore each row of $\tilde S$ is composed of the rows of exactly two $S_i$.

The bi-level distributed ALADIN  ADMM variant (ALADIN ADMM) relies on purely local communication; i.e.
 in each iteration, the respective $\lambda_i$'s between two neighboring regions are exchanged. Hence, in ALADIN ADMM one communicates 
$
2n_c \cdot n^{\mathrm{AD}}
$
floats locally, where $n^{\mathrm{AD}}$ is the number of inner ADMM iterations.

Similarly, in the bi-level distributed ALADIN with conjugate gradient (ALADIN CG) one communicates 
$
2n_c \cdot n^{\mathrm{CG}}
$
floats locally and additionally $2\cdot n_c^2$ in the local preparation phase (the rows of the Schur-complements $e_j^\top S_i$).
Finally, the global communication for computing $\alpha$ and $\beta$ is  $2N\cdot n^{\mathrm{CG}}$. 

\section{Numerical Case Studies} \label{sec:NumRes}

\subsection{AC Optimal Power Flow}
Non-convex  AC optimal power flow problems are of crucial interest in control of power systems. 
Specifically, we investigate the IEEE 30-bus system shown in Figure \ref{fig:ieee30bussystem} with data from \cite{Zimmerman2011}.  
For details on how to formulate OPF problems in form of \eqref{eq:sepProb} 
see \cite{Molzahn2017, Erseghe2015,Engelmann2019}. 
Here we use the problem formulation and partitioning $\mathcal{P}$ from \cite{Engelmann2019} with  ALADIN parameters $\rho = 10^6$, $\mu = 10^7$ and the step size for the lower-level ADMM $\rho= 2\cdot10^{-2}$.
In all cases we use warm-starting for CG and ADMM to accelerate convergence.

\begin{figure}[t]
	\centering
	\includegraphics[trim={7em 45em 7em 7em},clip,width=1\linewidth]{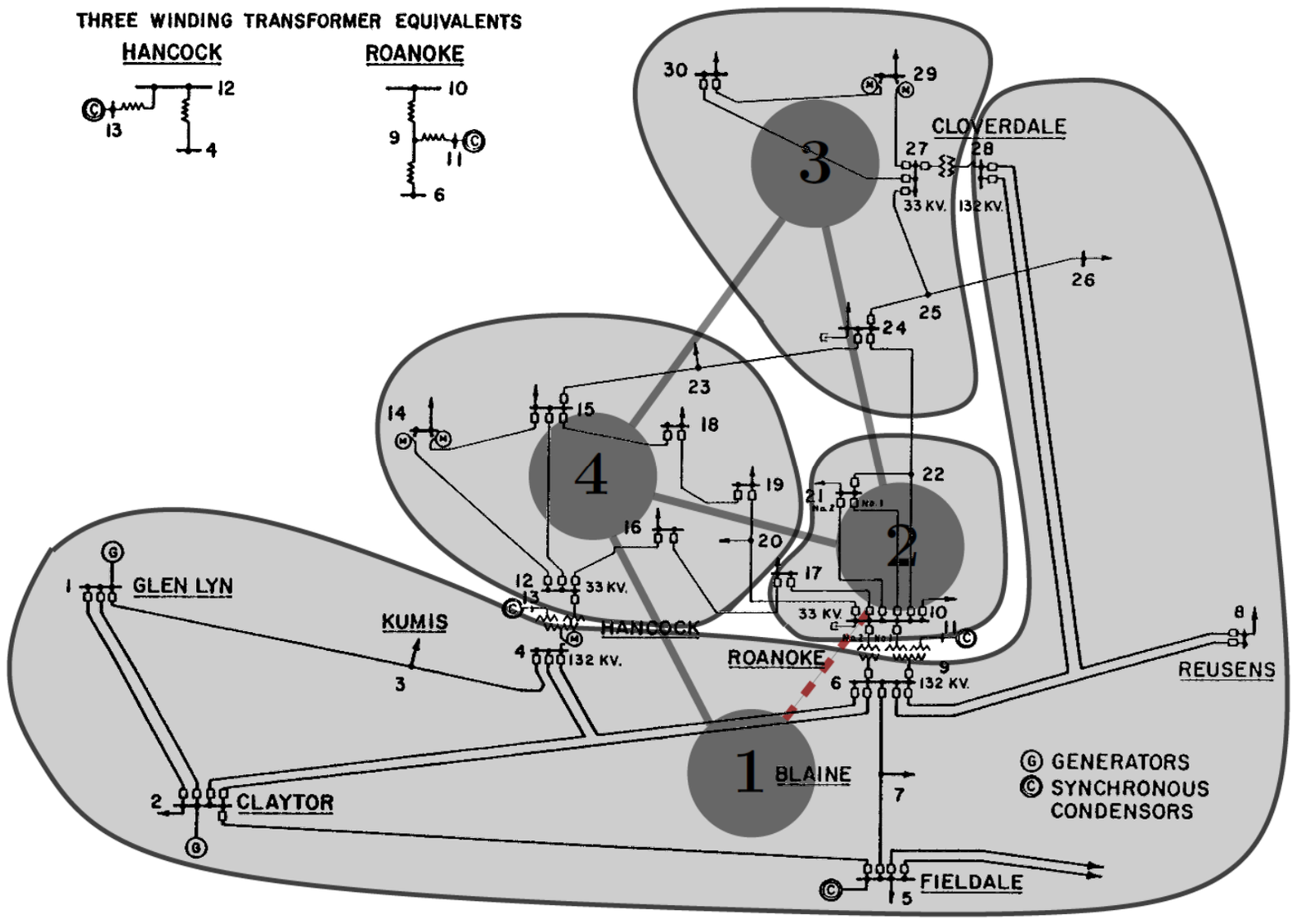}
	\caption{IEEE 30-bus system with induced connectivity graph $\mathcal{G}$ and paritioning $\mathcal{P}= \{\{1\text{-}8, 28\} $, $ \{9\text{-}11, 17, 21, 22\}$, $\{24\text{-}27, 29, 30\}$, $ \{12\text{-}16, 18\text{-}20, 23\}\}$.}
	\label{fig:ieee30bussystem}
\end{figure}

\begin{figure}[t]
	\centering
	\includegraphics[trim={1.4em 0.5em 0 0},clip,width=0.9\linewidth]{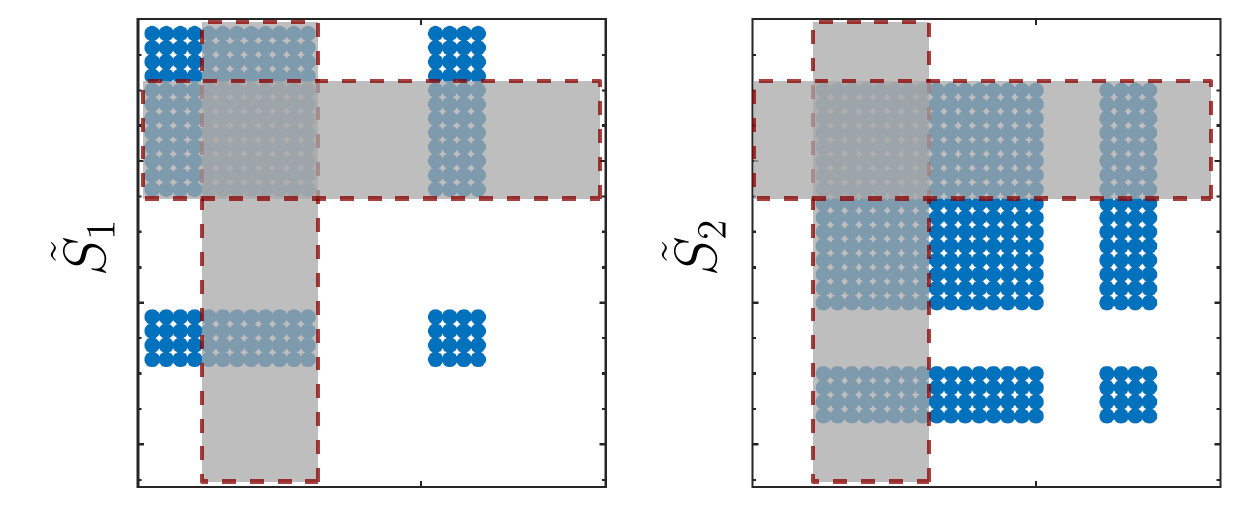}
	\caption{Sparsity patterns of the Schur complements for subproblem 1 and 2 with $\mathcal{C}(1)\cap\mathcal{C}(2)=\{5,\dots,12\}$ and the dashed interconnection from Figure \ref{fig:ieee30bussystem}.}
	\label{fig:Si}
\end{figure}

\begin{figure*}[t]
	\centering
	\begin{subfigure}{0.49\textwidth}
		\includegraphics[trim={2.5em 2em 0em 3em},width=1\linewidth]{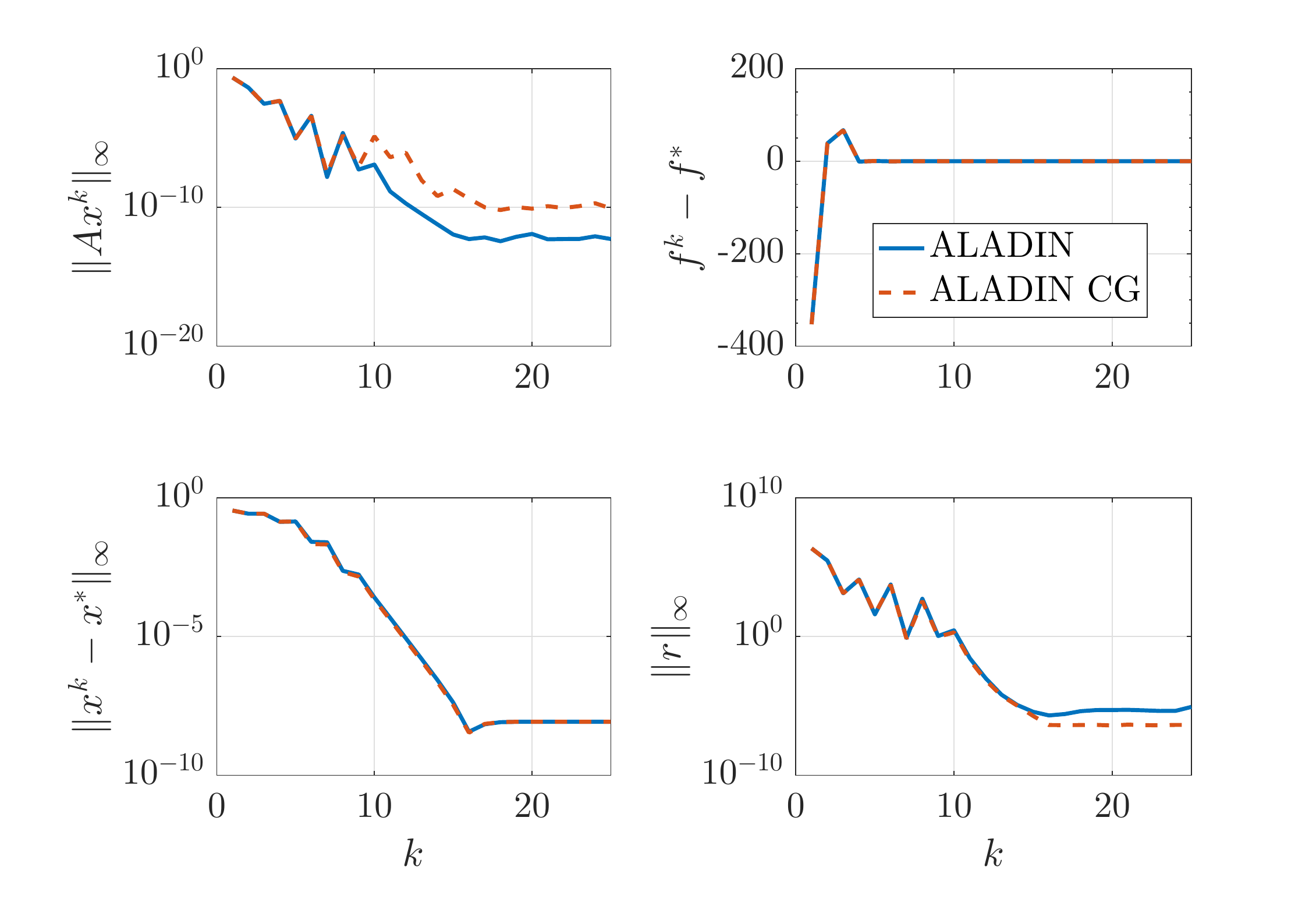}
		\captionsetup{width=.95\linewidth}
		\subcaption{Convergence of ALADIN with exact linear algebra and distributed conjugate gradient with $80$ inner iterations. \phantom{BlindtextBlindtextBlindtextBlindtext}}
		\label{fig:fullALCG}
	\end{subfigure}
	\begin{subfigure}{0.485\textwidth}
		\includegraphics[trim={2em 1em 0 3em},clip,width=1\linewidth]{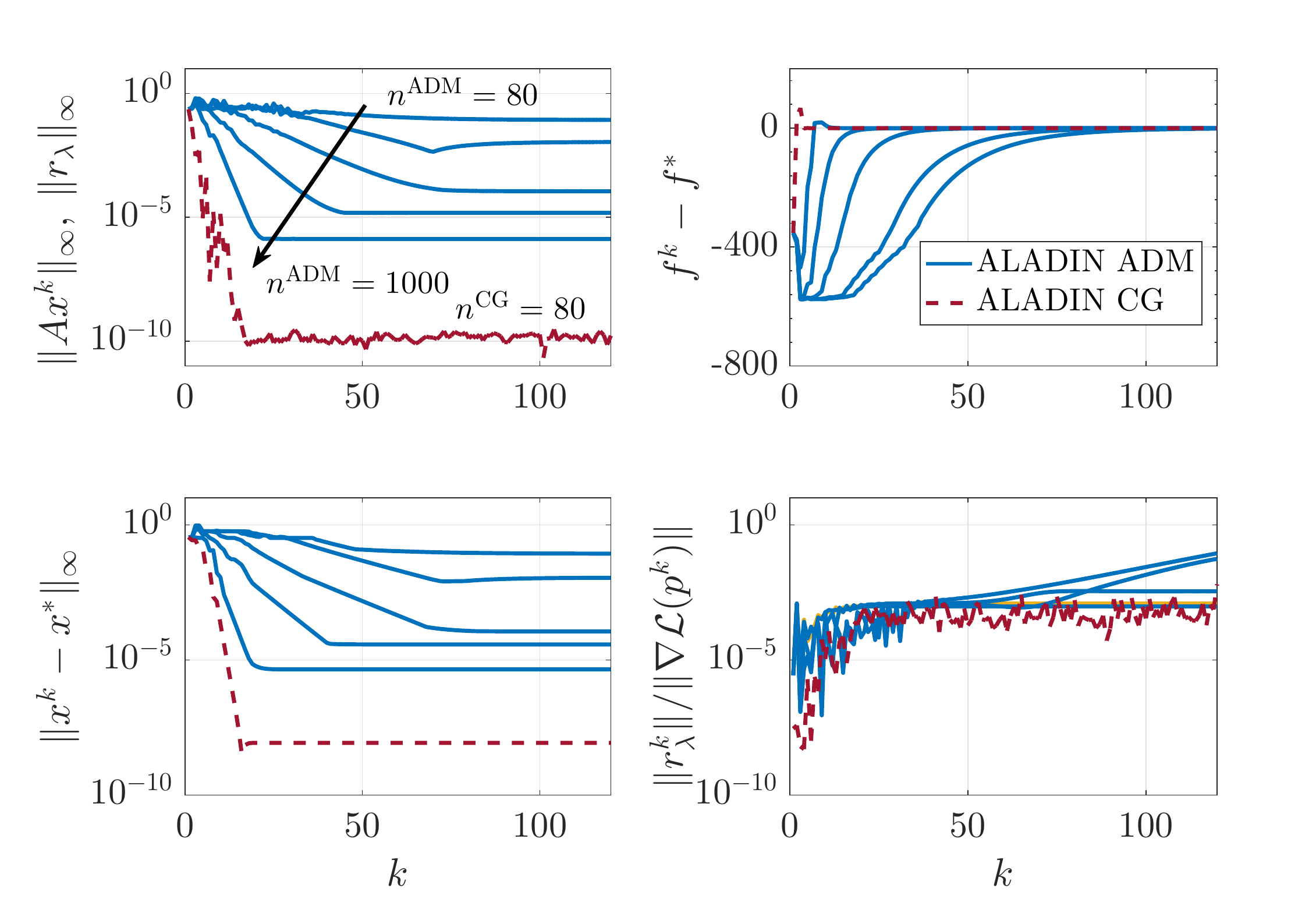}
		\captionsetup{width=1\linewidth}
		\subcaption{Convergence of ALADIN for $\{80,100,200,400,1000\}$ inner iterations with ADMM and for $80$ inner iterations with distributed conjugate gradient.}
		\label{fig:diffIters}
	\end{subfigure}
	\caption{Convergence behavior of different decentralized ALADIN variants.}
	\label{fig:convALCG}
\end{figure*}

 The 30-bus example has two physical interconnections between  subproblems 1 and 2 shown in Figure \ref{fig:ieee30bussystem}. 
 This leads to eight consensus constraints jointly assigned to subproblem 1 and 2 \cite{Engelmann2019}.
 Figure \ref{fig:Si} shows the resulting sparsity patterns of the corresponding  Schur-complements $ \tilde S_1\in \mathbb{R}^{32\times 32}$ and $ \tilde S_2\in \mathbb{R}^{32\times 32}$. 
 One can observe an overlap in the corresponding rows/columns of $\tilde S_1$ and $\tilde S_2$ predicted by Lemma \ref{lem:sparsity}. The rows/columns of the remaining Schur-complements $\tilde S_3$ and $\tilde S_4$ are zero respectively.

Figure \ref{fig:fullALCG}  shows the behavior of standard ALADIN (exactly solved coordination QP) and for ALADIN CG.  
Figure \ref{fig:diffIters} depicts the results for inexactly solved coordination QP with different fixed numbers of inner iterations for ALADIN CG and ALADIN ADMM. 
Observe that there is almost no difference in the convergence rate of standard ALADIN compared with ALADIN CG with 80 inner iterations.

In contrast, different numbers of inner iterations influence the total convergence behavior of ALADIN ADMM, cf. Figure \ref{fig:diffIters}. Indeed the convergence speed varies greatly with $n^{\text{AD}} \in \{80,100,200,400,1000\}$; also the achievable accuracy of ALADIN ADMM seems to be limited by different numbers of inner ADMM iterations.  
Whereas for ALADIN CG a fixed number of inner iterations yields good performance, the number of inner iterations necessary for ALADIN ADMM depends on the desired solution accuracy and it effects the overall convergence speed (i.e. the number of outer ALADIN iterations).

This behavior is underpinned by the total number of inner iterations (\# of inner iterations times \# of outer iterations) shown in Table \ref{tab::totInnerIter}. 
%
%

\begin{table}[t]
	\centering
	\renewcommand{\arraystretch}{1.7}
	\caption{Total iterations versus inner iterations for OPF.}
	\renewcommand{\arraystretch}{1.2}
	\begin{tabular}{crllllll}
		\toprule
		$n^{inner}$ & $\epsilon$ &	80 & 100  & 200  & 400 & 1000
		\\ \midrule
		CG & $10^{-4}$ & 800  &	800 & 800  &  800 & 800     \\ 
		ADMM & $10^{-2}$ & - & 7 000 & 7 000 & 7 600 & 11 000      \\
		& $10^{-3}$     & -        &	- & 10 800  &  10 800 & 13 000     \\
		& $10^{-4}$  & -  &	-    & -  &  14 800 & 16 000     \\ 
		\bottomrule
	\end{tabular}
	\label{tab::totInnerIter}
\end{table}

Figure \ref{fig:convLAlate} depicts the convergence behavior of distributed conjugate gradient and ADMM for two different instances of \eqref{eq:redRedSystemAlg}.
The left-hand side shows the results for ALADIN CG and ALADIN ADMM at one of the first iterations of ALADIN where $\tilde S$ is quite ill-conditioned. 
The right-hand side depicts the convergence of both algorithms when ALADIN is almost converged and therefore the condition number of $\tilde S$ is smaller. 
Observe the sublinear convergence rate of ADMM versus the practically superlinear convergence rate of conjugate gradient (cf. Table \ref{tab::convPropCGAL}) in both cases. 
Furthermore, note that the theoretical finite convergence of CG (here this would be 32 iterations) is not realized due to the conditioning of $\tilde S$. 
However, the practical convergence rate of centralized CG appears to be superior to most other available decentralized 
 methods \cite{Nedic2018}.

\begin{figure}[t]
	\centering
	\includegraphics[trim={0.5em 1em 0em 1em},width=0.48\linewidth]{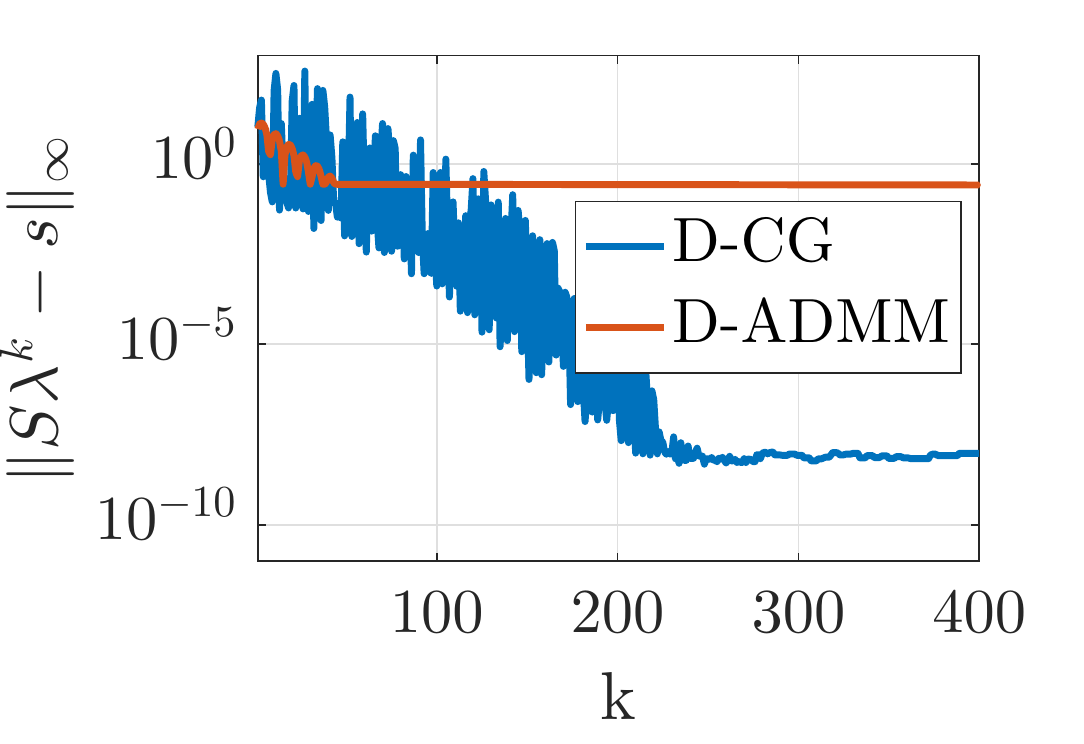}
	\includegraphics[trim={0.5em 1em 0em 1em},width=0.48\linewidth]{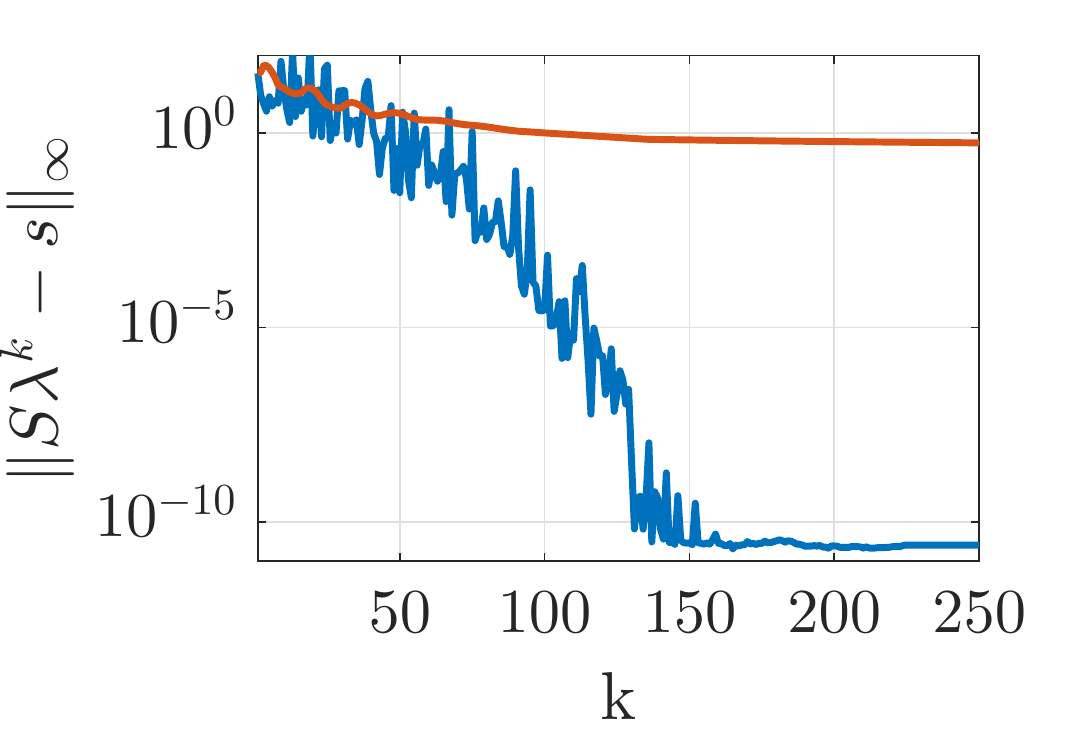}
	\caption{Convergence of ADMM  and CG for typical $\tilde S$ occuring in OPF. On the left we have a $\tilde S$ with $\operatorname{cond(\tilde S)=8 \cdot 10^9}$ in the beginning of the ALADIN iterations and on the right we have $\operatorname{cond(\tilde S)=5 \cdot 10^6}$ when ALADIN is almost converged.}
	\label{fig:convLAlate}
\end{figure}

%
%
%
%
%

\subsection{Distributed control of mobile robots}
As a second example  we consider an Optimal Control Problem (OCP) where two mobile robots should reach their final position while keeping a minimum distance to each other, cf. \cite{Mehrez2017}.
The centralized OCP reads
\begin{subequations} \label{eq:robotOCP}
\begin{align} 
&\hspace{-2em}\min_{z_i(\cdot),u_i(\cdot), \forall i \in \mathcal{R}} \int_0^T \sum_{i\in \mathcal{R}} \|z_i-z^e_i\|_{Q_i}^2 + \|u_i\|_{R_i}^2\, dt \label{eq:stageCost} \\
\quad \text{s.t.} \;\;\; & \dot z_i(t) = f_i(z_i(t),u_i(t)),\;\;  z_i(0)=z_{i0}, && \hspace{-0.6cm}\forall i \in \mathcal{R} \label{eq:dynamics} \\ \label{eq:zeroTermConstr}
&(x,y)_i^\top(T)=(x^e,y^e)_i^\top, &&\hspace{-0.6cm} \forall i \in \mathcal{R} \\
& \|(x, y)_i^\top(t)-(x, y)_j^\top(t)\|_2^2\geq d^2, && \hspace{-0.6cm} i \neq j
\end{align}
\end{subequations}
where $z_i=(x_i\; y_i\; \theta_i)^\top$ is the state of each robot $i \in \mathcal{R}$, $x_i$ and $y_i$ describe the robots position in the $x$-$y$-plane, and $\theta_i$ is the yaw angle with respect to the $x$-axis (Fig. \ref{fig:robots}). 
The stage cost \eqref{eq:stageCost} is the sum of quadratic tracking cost with respect to the desired end position $z^e_i \in \mathbb{R}^3$ for all robots. Constraint  \eqref{eq:dynamics} are the continuous-time dynamics  
\[
\dot z_i= f_i(z_i,u_i) :=
\begin{pmatrix}
v_i\cos (\theta_i)  & v_i\sin(\theta_i) & \omega_i \end{pmatrix}^\top.
\forall i \in \{1,2\}.
\]
The inputs $u_i=(v_i\;\omega_i)^\top$ are the velocity $v_i$ the turning rate $\omega_i$.
The terminal constraint \eqref{eq:zeroTermConstr}  and the stage cost \eqref{eq:stageCost} are chosen having  
a distributed NMPC setting in mind \cite{Rawlings2017}.

In order to convert \eqref{eq:robotOCP} into a partially separable NLP \eqref{eq:sepProb}, we introduce auxiliary variables duplicating the predicted $(x$-$y)$ trajectories of each robot pair and enforce consensus by the constraint \eqref{eq:ConsConstr}. 
Due to space limitations we do not elaborate this in detail.
We employ a direct solution approach and discretize \eqref{eq:robotOCP} via Euler-backward; the sampling period is $0.1\,$ seconds and the horizon is $T=10\,$ seconds. We consider $|\mathcal{R}|=2$ robots which should keep a distance of $d=5\,$m with  $Q=0.1\cdot \operatorname{diag}\large ((10\;\;10\;\;1))$ and $R=\operatorname{diag}\large ((1\;\;1))$.
We use $\rho=10^2$, $\mu = 10^6$ and $\rho^{\text{AD}}=10^{-1}$ as tuning parameters for ALADIN.

Figure \ref{fig:openLoppTraj} shows the optimal open-loop trajectories for \eqref{eq:robotOCP}.
One can observe that the goal of collision avoidance is satisfied
while the robots move to their target positions.
Interestingly, Problem \eqref{eq:robotOCP} seems to be numerically quite different to the OPF problem.
Here, $n^{\text{CG}}=30$ inner iterations for CG suffice for local convergence although the problem size is ($n_x=1\,200$) much larger. 
At the same time, at least $n^{\text{AD}}=2\,400$ inner iterations were needed for ADMM to achieve an accuracy of $\epsilon=10^{-4}$. 

\begin{figure}[t]
	\centering
	\captionsetup[subfigure]{position=b}
	\subcaptionbox{Robot models. 	\label{fig:robots}}{\includegraphics[width=0.49\linewidth]{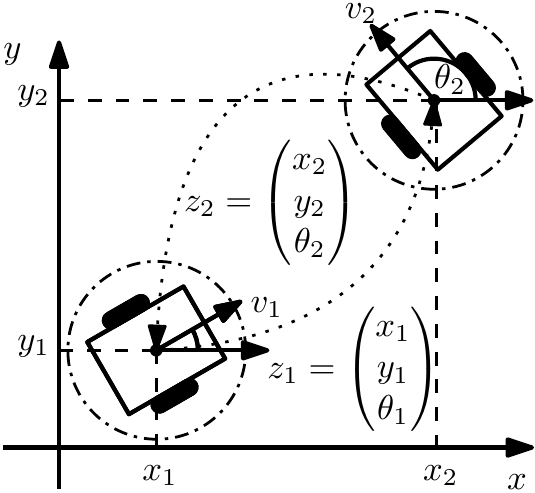}}
	\centering
	\subcaptionbox{Open-loop trajectories. 	\label{fig:openLoppTraj}}{\includegraphics[trim={0em 0em 0em 0em},clip,width=0.49\linewidth]{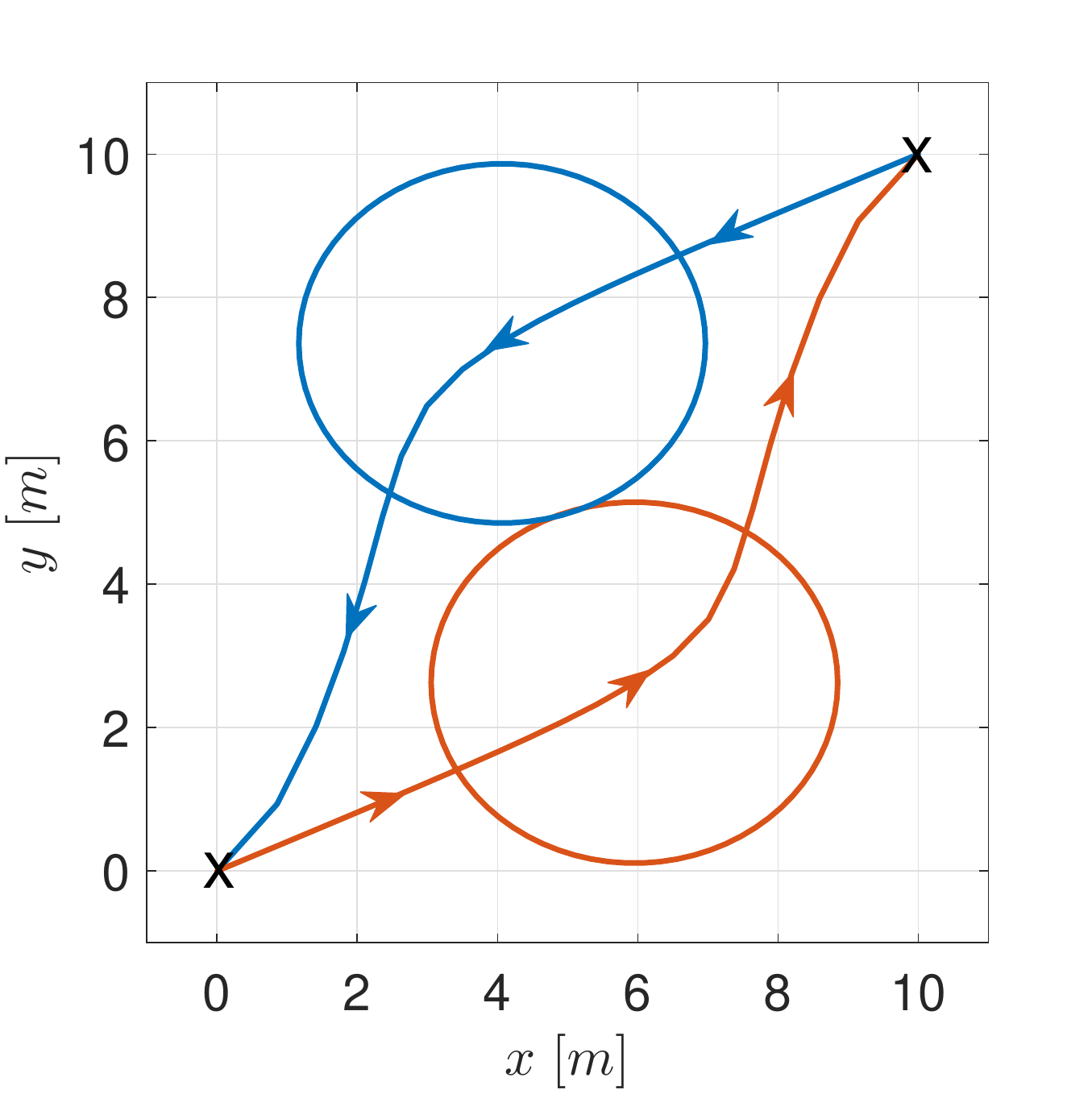}}
	\caption{Distributed control of mobile robots.}
\end{figure}

\subsection{Numerical communication analysis}
Finally, we evaluate forward communication as introduced in Section \ref{se:commAnal} practically. 
Table \ref{tab::comparison} summarizes  the forward communication for both examples. 
In addition the last two rows in both parts of Table~\ref{tab::comparison} depict the total communication (per step-communication times outer \# of ALADIN iterations) for a termination tolerance of $\epsilon=10^{-4}$.

\begin{table}
	\centering
	\renewcommand{\arraystretch}{1.7}
	\caption{Forward comm. to $\epsilon = 10^{-4}$ with $n^{\text{AD}}=400$, $n^{\text{CG}}=80$ for OPF and $n^{\text{AD}}=2\,400$, $n^{\text{CG}}=30$ for robot control.}
	\renewcommand{\arraystretch}{1.2}
	\begin{tabular}{crcccc}
		\toprule
		&      variant &                    standard                   & \hspace{-1em} condensed &    ADMM    &    CG     \\ \midrule
		\multirow{5}{*}{{\rotatebox[origin=c]{90}{OPF}}}
		&  local prep. &                     -                      &            -            &     -      & $2\,048$  \\
		& local  iter. &                     -                      &            -            & $25\, 600$ & $5\,120$  \\
		&       global &  \hspace{-1em} $>\hspace{-0.3em}9\,858$  & \hspace{-1em}  $1\,056$ &     -      &   $960$   \\
		\cmidrule{2-6} &   local tot. &                     -                      &            -            & $691\,200$ & $53\,248$ \\
		&  global tot. & \hspace{-1em} $>\hspace{-0.3em} 98\,580$ & \hspace{-1em} $10\,560$ &     -      & $9\,600$  \\ \midrule
		\multirow{5}{*}{{\rotatebox[origin=c]{90}{robots}}}
		&  local prep. &                     -                      &            -            &     -      & $80\,000$ \\
		& local  iter. &                     -                      &            -            &   $960\,000$   & $12\,000$ \\
		&       global &     \hspace{-1em} $>\hspace{-0.3em}824\,506$     &  $40\,200$    &     -      &   $120$   \\
		\cmidrule{2-6} &   local tot. &              -            &            -            &     $9\,600\,$k       & $200\,$k  \\
		&  global tot. &    \hspace{-1em} $>\hspace{-0.3em} 20\,613\,$k     &     $1\,005\,$k     &     -      &  $1\,$k   \\ \bottomrule
	\end{tabular}
	\label{tab::comparison}
\end{table}

As expected, ALADIN with condensing (Algorithm \ref{alg:ALADINbil}) needs much less communication compared to standard ALADIN variant (Algorithm \ref{alg:ALADIN2}). 
Solving \eqref{eq:redRedSystemAlg} with the decentralized variants of conjugate gradient or ADMM increases total communication compared to the condensed ALADIN variant. 
Furthermore, the total communication of ALADIN CG is smaller compared to standard ALADIN. 
The comparably large local communication burden of ALADIN ADMM stems from the increased number of inner iterations, cf. Figure \ref{fig:diffIters} and Table \ref{tab::totInnerIter}. 

Finally, it is worth to be noted investing the very limited global coordination and communication effort required by ALADIN CG one can  achieve much better performance compared with entirely decentralized coordination, cf. right-hand side columns of Table \ref{tab::comparison}. 


\section{Summary \& Outlook}
This paper has proposed a framework for designing decentralized algorithms for non-convex constrained optimization problems via bi-level distribution of the ALADIN algorithm. The core idea is to add a second (inner) layer of distributed/decentralized computation to ALADIN, whereby the coordination QP is first condensed (as a post-processing step of solving the local non-convex subproblems) and then solved in decentralized fashion. 
We have presented sufficient conditions on the numerical solution accuracy necessary to preserve local quadratic convergence properties of ALADIN. Moreover, we have shown how this bound can be enforced by means of decentralized inner algorithms. Specifically, we have proposed  a decentralized variant of the conjugate gradient method, which shows promising performance. We also compared it to using ADMM at the inner level. 
Simulation studies from power systems and robotics underpin the efficacy of the proposed scheme. These studies also indicate that decentralized conjugate gradient outperforms ADMM in terms of convergence speed and in terms of total communication effort. 

We expect that the proposed bi-level distribution framework opens new avenues for future research, e.g.,  on decentralizing globalization strategies or on tailored decentralized algorithms for distributed non-linear model predictive control.

\appendices

\section{}
\begin{lemma}[$S_i \succeq 0, S \succ 0$ and $S_i^\top=S_i$] \label{lem:SposDef}
	The matrices $S_i$ are positive semidefinite, $S=\sum_{i=1}^N S_i$ is positive definite and $S_i=S_i^\top$.
\end{lemma}
\begin{proof}
	By Assumption \ref{ass:LICQ} we have that all $\bar H_i$s are positive definite, i.e. $x^\top \bar H_i x > 0$ for all $x \in \mathbb{R}^{n_i}$. With $x:=H_i^{-1}y$ we have $x^\top \bar H_i x=y^\top (\bar H_i^{-1})^\top \bar H_i \bar H_i^{-1}y=y^\top \bar H_i^{-1} y > 0$. Furthermore let $y:=\bar A_i z$. Then $z^\top \bar A_i^\top \bar H_i^{-1} \bar A_i z = z^\top S_i z  \geq 0$ for all $z \in \mathbb{R}^{n_c}$ as $\bar A_i$ may be rank deficient.  
	
	$S\succ 0:$ We know from Assumption \ref{ass:LICQ} that $x^\top \bar Hx > 0$. By defining $y:=Ax$ we have $x^\top A^\top \bar H A x=x^\top Sx >0$ as $A$ has full rank by LICQ.
	
	As $H_i=H_i^\top$, we have $\bar H_i^\top = (Z_i^\top H_i Z_i)^\top=(H_i Z_i)^\top Z_i = Z_i^\top H_i^\top Z_i= \bar H_i$ and by the same argument $S_i^\top=(\bar A_i \bar H_i^{-1}\bar A_i^\top)^\top=S_i$. 
	To obtain $\tilde S_i$ we add elements to the main diagonal only yielding $\tilde S_i=\tilde S_i^\top$.
\end{proof}

\printbibliography



\ifCLASSOPTIONcaptionsoff
  \newpage
\fi

\end{document}